\documentclass[reqno]{amsart}

\usepackage{amsfonts}
\usepackage{amssymb}

\usepackage{amscd}
\usepackage{pictexwd,dcpic}
\usepackage{graphicx}
\usepackage{tikz}
\usepackage{fullpage}
\usepackage{hyperref}
\hypersetup{
    colorlinks=true,
    linkcolor=blue,
    filecolor=magenta,
    urlcolor=cyan,
    citecolor=cyan,}
\usepackage[nameinlink,capitalize]{cleveref}
\usepackage{enumitem}

\title{On Rees algebras and de Jonqui\`eres transformations}

\author{Matthew Weaver}
\address{School of Mathematical and Statistical Sciences, Arizona State University, Wexler Hall, Tempe AZ 85281}
\email{matthew.j.weaver@asu.edu}

\date{}	

\newtheorem{thmx}{Theorem}

\newtheorem{thm}{Theorem}[section]
\newtheorem*{thm-nonum}{Theorem}

\newtheorem{prop}[thm]{Proposition}
\newtheorem{lemma}[thm]{Lemma}
\newtheorem{cor}[thm]{Corollary}
\numberwithin{equation}{section}

\theoremstyle{definition}
\newtheorem{rem}[thm]{Remark}
\newtheorem{set}[thm]{Setting}
\newtheorem{notat}[thm]{Notation}
\newtheorem{defn}[thm]{Definition}

\newtheorem{ex}[thm]{Example}

\Crefname{thm}{Theorem}{Theorems}
\Crefname{ex}{Example}{Examples}

\def\A{\mathcal{A}}

\def\D{\mathcal{D}}
\def\G{\mathcal{G}}
\def\I{\mathcal{I}}
\def\J{\mathcal{J}}

\def\K{\mathcal{K}}
\def\L{\mathcal{L}}
\def\m{\mathfrak{m}}
\def\n{\mathfrak{n}}

\def\p{\mathfrak{p}}
\def\P{\mathbb{P}}

\def\R{\mathcal{R}}
\def\S{\mathcal{S}}
\def\F{\mathcal{F}}

\def\x{\underline{x}}

\def\Crem{\mathfrak{C}}
\def\deJonq{\mathfrak{J}}
\def\Cremdeg{\mathfrak{d}}

\def\coker{\mathop{\rm coker}}
\def\dim{\mathop{\rm dim}}
\def\depth{\mathop{\rm depth}}

\def\hgt{\mathop{\rm ht}}

\def\spec{\mathop{\rm Spec}}

\def\bideg{\mathop{\rm bideg}}

\usepackage[all]{xy}

\newcommand*{\dit}{{\scalebox{0.5}{$\bullet$}}}

\usepackage{nicematrix}

\begin{document}

\begin{abstract}
We recall a higher dimension analog of the classic plane de Jonqui\`eres transformations, as given by Hassanzadeh and Simis in \cite{HS14}. Such a parameterization defines a birational map from $\P^{n-1}$ to a hypersurface in $\P^{n}$, and a natural question that arises is how to obtain its implicit equation. We pass from the image of this map to its graph, and implicitize the Rees algebra of the ideal of the de Jonqui\`eres map when its underlying Cremona support is tame. We then consider the Rees rings of ideals of generalized de Jonqui\`eres transformations, as introduced in \cite{RS22}, and answer a conjecture of Ramos and Simis.
\end{abstract}

\maketitle


\section{Introduction}

The study of birational transformations is classically rooted in algebraic geometry with a rich history. In the early 20th century, the generation of the \textit{Cremona group} of $\mathbb{P}^2$, consisting of all such birational transformations of the projective plane, was a topic of great interest for both geometers and algebraists alike. In a seminal result, Noether showed that such a transformation could be expressed as a composition of quadratic transformations. Castelnuovo later showed that plane Cremona transformations could also be decomposed into \textit{de Jonqui\`eres} transformations. Even in $\P^3$, the Cremona group and related constructions become significantly more complicated, however there are adaptations of the above notions to higher dimensions \cite{HS14,Pan01,PS15,RS22}. 

The notion of a de Jonqui\`eres transformation in higher dimension typically relies on the existence of a supporting Cremona transformation in one dimension lower. In the case of plane de Jonqui\`eres transformations, this underlying Cremona map is the identity, but this fails to be the case beyond transformations of $\P^2$. As such, the nature of most higher dimensional analogs are elusive in general. In this article, we consider two such adaptions of this notion, namely the de Jonqui\`eres transformation $\P^{n-1} \dashrightarrow \P^{n}$ as introduced by Hassanzadeh and Simis in \cite{HS14}, and the later extension of a \textit{generalized} de Jonqui\`eres transformation $\P^{n} \dashrightarrow \P^n$ as introduced by Ramos and Simis in \cite{RS22}.

Throughout we assume that $k$ is an algebraically closed field, and consider a rational map $\Psi:\, \P^n\dashrightarrow \P^m$ defined by homogeneous polynomials $f_0,\ldots,f_m \in R=k[x_0,\ldots,x_n]$ of the same degree. The closed image of $\Psi$ is a subvariety $X\subseteq \P^m$, with homogeneous coordinate ring $k[f_0,\ldots,f_m]$ as a $k$-subalgebra of $R$. However, as $X$ is a subvariety of $\P^m$, there exists an ideal $\I(X)$ of $S=k[y_0,\ldots,y_m]$, the coordinate ring of $\P^m$, such that $k[f_0,\ldots,f_m] \cong S/\I(X)$. In this sense, the polynomials $f_0,\ldots,f_m$ are commonly seen as a \textit{parameterization} of $X$, and the ideal $\I(X)$ in $k[y_0,\ldots,y_m]$ as the \textit{implicitization} of $X$. Determining the generators of $\I(X)$ is the so-called \textit{implicitization problem}.

In the setting above, it is often more advantageous to implicitize the \textit{graph} of $\Psi$, as this carries much more information, and then project onto the image. Recall that the coordinate ring of the graph is the \textit{Rees algebra} of the ideal $I = (f_0,\ldots,f_m)$, which is the graded subring $\R(I) = R[f_0t,\ldots,f_mt]$ of $R[t]$, for $t$ an indeterminate.  There is a natural epimorphism $\Psi: R[y_0,\ldots,y_m] \rightarrow \R(I)$ given by $y_i\mapsto f_i t$, which induces an isomorphism $\R(I) \cong R[y_0,\ldots,y_m]/\J$, where $\J= \ker \Psi$ is the \textit{defining ideal} of $\R(I)$. With this, we note that $\R(I) /\m \R(I) \cong S/\I(X)$, where $\m= (x_0,\ldots,x_n)$, hence it is enough to determine the implicit equations of $\R(I)$.
The search for generators of $\J$ has been well studied, yet a full set of \textit{defining equations} of the Rees algebra is unknown in most settings. There has been success in particular settings when the ideal $I$ has low codimension (see e.g. \cite{BM16,CPW23,Morey96,MU96,Nguyen14,Nguyen17,Weaver23,Weaver24}), but the problem remains open in general.

In this paper, we consider the implicitization of de Jonqui\`eres transformations, taking all conventions from \cite{HS14}, and study the Rees algebra of the ideals defining these rational maps. We proceed along a path similar to the one taken in \cite{BM16} and employ the \textit{method of divisors}, as introduced in \cite{KPU11}, to determine the defining ideal. We consider the Rees ring of the ideal of a de Jonqui\`eres transformation when its underlying Cremona support is as tame as possible, namely the identity. The main result, a formulation of \Cref{defining ideal theorem} and \Cref{Cohen-Macaulayness of Rees algebra}, is as follows.

\begin{thmx}\label{intro theorem A}
Let $\deJonq:\,\P^{n-1}\dashrightarrow \P^n$ be a de Jonqui\`eres transformation with Cremona support $\Crem$ the identity. Let $I=(fx_1,\ldots,fx_n,g) \subseteq R=k[x_1,\ldots,x_n]$ denote the ideal of this rational map for homogeneous polynomials $f$ and $g$ with $\deg g = \deg f+1=d\geq 2$. The Rees algebra of $I$ is $\R(I) \cong R[y_1,\ldots,y_{n+1}]/\J$ where
$$\J = I_2(\psi) + (h_1,\ldots,h_d)$$
where $\psi$ is a $2\times n$ matrix of indeterminates, and $h_1,\ldots,h_d$ is a downgraded sequence of polynomials associated to $f$ and $g$. Additionally, $\R(I)$ is almost Cohen-Macaulay and is Cohen-Macaulay if and only if $d\leq n$. 
\end{thmx}

The essential ingredient to the proof of \Cref{intro theorem A} is the notion of a \textit{downgraded sequence} of a particular syzygy associated to $f$ and $g$. This recursive algorithm produces equations of the defining ideal at each iteration, and yields a full generating set in the setting above. Our description of this process is borrowed from similar processes used in \cite{HS14,RS22}, but we also note that similar methods have been used to produce equations of the Rees ring in other instances \cite{BM16,CHW08,Weaver23,Weaver24,Weaver25}. Once the defining ideal $\J$ has been determined, the Cohen-Macaulayness of $\R(I)$ is studied by relating $\J$ to an ideal whose powers have resolutions within the family of Eagon-Northcott complexes.

Once Rees rings of ideals of de Jonqui\`eres transformations have been studied, we show that the methods used in the proof of \Cref{intro theorem A} can be adapted to the setting of \textit{generalized} de Jonqui\`eres transformations $\P^{n} \dashrightarrow \P^n$, as introduced in \cite{RS22}. In the case of identity support, the equations of the Rees algebra are known \cite[2.6]{RS22}, and may be realized as a downgraded sequence. However, it was left as a question as to when the Rees ring is Cohen-Macaulay \cite[2.10]{RS22}. Applying the methods of \Cref{intro theorem A} to this setting, we show that the Cohen-Macaulayness of $\R(I)$ may be determined. The main result for generalized de Jonqui\`eres transformations, \Cref{gen de Jonq - Cohen-Macaulayness of Rees algebra}, is as follows.

\begin{thmx}\label{intro Theorem B}
Let $\deJonq:\,\P^{n} \dashrightarrow \P^n$ be a generalized de Jonqui\`eres transformation with Cremona support $\Crem$ the identity. Let $I=(fx_1,\ldots,f x_n,g) \subseteq R=k[x_1,\ldots,x_{n+1}]$ denote the ideal of this transformation for homogeneous polynomials $f$ and $g$ with $\deg g = \deg f+1=d\geq 2$. The Rees algebra $\R(I)$ is almost Cohen-Macaulay and is Cohen-Macaulay if and only if $d\leq n+1$. 
\end{thmx}

In particular, this answers a conjecture of Ramos and Simis \cite[2.10]{RS22} in the affirmative. As noted, the defining equations of the Rees ring are already known \cite[2.6]{RS22}, hence we are only interested in the Cohen-Macaulay property of $\R(I)$ in \Cref{intro Theorem B}. However, we note that the methods presented here can be used to provide an alternate proof for the generation of the defining ideal by a downgraded sequence.

We now briefly describe how this article is organized. In \Cref{Main Section}, we recall the notion of a de Jonqui\`eres transformation as in \cite{HS14}, and begin our study of the Rees ring of the ideal of such a rational map in the case of identity support. In \Cref{approximation section}, we introduce a simpler algebra defined by a determinantal ideal which maps onto the Rees algebra, and we compare properties of these rings. In \Cref{Algorithm Section}, we present the method of downgraded sequences, similar to the techniques used in \cite{HS14,RS22}, and show that this process yields equations of the Rees algebra. In \Cref{Defining Ideal Section}, we show that this method produces a complete generating set of the defining ideal, and use this description to describe the Cohen-Macaulayness of the Rees algebra. In \Cref{Generalized de Jonq section}, we recall the notion of a \textit{generalized} de Jonqui\`eres transformation, as given in \cite{RS22}. We show that many of the constructions developed in the previous sections are easily adapted to this setting. In particular, we determine the Cohen-Macaulayness of the Rees ring in a manner similar to the procedure of \Cref{Defining Ideal Section}, and answer the aforementioned conjecture of Ramos and Simis.


\section{Rees algebras and de Jonqui\`eres maps}\label{Main Section}

We begin by recalling the notion of a \textit{de Jonqui\`eres transformation} $\deJonq:\, \P^{n-1}\dashrightarrow \P^n$, as given by Hassanzadeh and Simis in \cite{HS14}. We note that these maps and the Rees rings of their ideals have been studied in generality, hence we begin by recounting a few items at the start of this section. We then study the Rees ring of the ideal of forms defining this map, when its Cremona support is mild.

Recall that a rational map $\Psi:\, \P^n\dashrightarrow \P^m$ is said to be \textit{birational} onto its image $X$ if there exists a reverse map $\Psi':\, \P^m\dashrightarrow \P^n$, which gives a proper inverse when restricted to $X\subseteq \P^m$. In particular, we have the special case of a \textit{Cremona transformation}, namely a birational map $\Crem:\,\P^n\dashrightarrow \P^n$ of $\P^n$ onto itself. With this, we may construct a de Jonqui\`eres transformation $\deJonq:\, \P^{n-1}\dashrightarrow \P^n$, as given in \cite{HS14}. Writing $R=k[x_1,\ldots, x_n]$ to denote the coordinate ring of $\P^{n-1}$, let $\Crem:\,\P^{n-1}\dashrightarrow \P^{n-1}$ denote a Cremona transformation defined by homogeneous forms $g_1,\ldots,g_n$ of degree $\Cremdeg \geq 1$ in $R$.

\begin{defn}[{\cite{HS14}}]\label{de Jonquieres defn}
Let $f$ and $g$ be two additional homogeneous polynomials of $R=k[x_1,\ldots,x_n]$ with $\deg g = \deg f+\Cremdeg$ and $\gcd(f,g) =1$. The rational map $\deJonq:\, \P^{n-1}\dashrightarrow \P^n$ defined by $fg_1,\ldots,fg_n,g$ will be called a \textit{de Jonqui\`eres transformation}, with Cremona support $\Crem$.
\end{defn}

An important note is that the map $\deJonq$ above is birational onto its image \cite[1.11]{DHS12}, which is a hypersurface $V(F)\subseteq \P^n$ for some homogeneous polynomial $F$ in $k[y_1,\ldots,y_{n+1}]$, the coordinate ring of $\P^n$.

We now begin our study of the implicitization of a de Jonqui\`eres transformation when the underlying Cremona support is the identity. Our approach will be purely algebraic, but the geometric notions may be easily recovered, as noted above. Our primary setting is the following.

\begin{set}\label{main setting}
Let $R=k[x_1,\ldots,x_n]$ for $n\geq 2$, with homogeneous maximal ideal $\m=(x_1,\ldots,x_n)$. Let $f$ and $g$ be homogeneous polynomials of $R$, with $\deg g=d$ and $\deg f=d-1$, for some integer $d\geq2$. Further assume that $\gcd(f,g)=1$ and let $I=(f\m,g) = (fx_1,\ldots,fx_n,g)$.
\end{set}

From the previous definition, $I$ is the ideal of the de Jonqui\`eres transformation $\deJonq$ defined by $fx_1,\ldots,fx_n,g$, with supporting Cremona map the identity $\Crem:\, \P^{n-1}\dashrightarrow \P^{n-1}$ given by $x_1,\ldots,x_n$. Notice that $I$ is an ideal of codimension two as $V(I) = V(f,g)\cup \{\m\}$, and $f,g$ is a regular sequence since $\gcd(f,g)=1$.
Rees rings of ideals with codimension two have been studied extensively (see e.g. \cite{BM16,MU96,Nguyen14,Nguyen17,Weaver23}), however most cases require the assumption of perfection, so as to use the Hilbert-Burch theorem \cite[20.15]{Eisenbud}. Whereas this will not be the case here (unless $n=2$), a presentation of this ideal is readily available from the mapping cone construction.

\begin{prop}[{\cite[2.2]{HS14}}] \label{Presentation of I}
With respect to the generating set $\{fx_1,\ldots,fx_n,g\}$, the ideal $I$ may be presented by 
$$R^{\binom{n}{2} +1} \overset{\varphi}{\longrightarrow} R^{n+1} \rightarrow I \rightarrow 0$$ 
where
\[
\varphi=\left(
\begin{array}{ccc|c}
   &  & & \\
   & \delta_1(\x)& & \partial g\\
   & & & \\
   \hline
   0& \cdots & 0 &-f 
\end{array}
\right)
\]
where $\delta_1(\x)$ is the first Koszul differential on $\x= x_1,\ldots,x_n$, and $\partial g$ is a column of homogeneous elements, of the same degree, with $[x_1 \ldots x_n]\cdot \partial g= g$.
\end{prop}

We note that $\partial g$ is not unique in general, and there are often many choices. However, any such choice will do for our purposes and the presentation above. We will revisit this construction in \Cref{Algorithm Section}.

\begin{rem}\label{S(I) remark}
Before we begin our study of the Rees ring of $I$, we introduce an algebra closely related to it, namely the \textit{symmetric algebra} $\S(I)$. With the presentation $\varphi$ above, recall that $\S(I) \cong R[y_1,\ldots,y_{n+1}]/\L$ where $\L$ is the ideal of entries $\L = ([y_1,\ldots,y_{n+1}]\cdot \varphi)$. Moreover, notice that the ideal of entries of the matrix product $[y_1\ldots y_n]\cdot \delta_1(\x)$ agrees with $I_2(\psi)$
 where
\begin{equation}\label{psi defn}
\psi =\begin{bmatrix}
    x_1&\ldots &x_n\\
    y_1&\ldots& y_n
\end{bmatrix}
\end{equation}
and so $\L = I_2(\psi) + (h)$, where $h$ is the product of $[y_1\ldots y_{n+1}]$ and the last column of $\varphi$.
\end{rem}

Writing $\J$ to denote the defining ideal of the Rees ring $\R(I)$ as in the introduction, recall that there is a natural epimorphism $\S(I) \rightarrow \R(I)$. As such, $\L \subseteq \J$ and we will make use of this containment to compare these two ideals.

\begin{prop}\label{J a saturation}
With $\J$ the defining ideal of $\R(I)$ and $\L$ as above, we have $\J = \L:\m^\infty$.
\end{prop}

\begin{proof}
As $R$ is a domain, we have that $\R(I)$ is a domain and so $\J$ is a prime ideal. With this and noting that $\m \nsubseteq \J$ and $\L\subseteq \J$, it is clear that $\L:\m^\infty \subseteq \J$. For the reverse containment, notice that for any prime ideal $\p\in \spec(R)\setminus\{\m\}$, we have $I_\p = (f,g)_\p$. As noted, $f, g$ is a regular sequence, hence $(f,g)_\p$ is of linear type. Thus we see that $\L_\p = \J_\p$ for any such prime $\p$, and so the quotient $\A=\J/\L$ is supported only at $\m$. Thus $\A$ is annihilated by some power of $\m$, which shows that $\J \subseteq \L:\m^\infty$.
\end{proof}

With the description of $\J$ above, we note that it will be more beneficial to describe this saturation as colon ideal for a specific power of $\m$. First however, we make a few observations on the symmetric algebra $\S(I)$.

\begin{prop}\label{S(I) dimension}
With the assumptions of \Cref{main setting}, we have $\dim \S(I) = n+1$. In particular, $\hgt \L = n$.
\end{prop}

\begin{proof}
We make use of the well-known formula of Huneke and Rossi \cite[2.6]{HR86}. Recall that 
$$\dim \S(I) = \sup \big\{ \mu(I_\p) + \dim R/\p \hspace{1mm}\big|\hspace{1mm}\p\in \spec (R)\big\}$$
where $\mu(-)$ denotes minimal number of generators. As such, we compute the value of $\mu(I_\p) + \dim R/\p$ in a few cases. If $\p\nsupseteq I$, then $I_\p$ is the unit ideal and so $\mu(I_\p) + \dim R/\p \leq n+1$. Now suppose that $\p\supseteq I$, and note that then $\hgt \p \geq 2$. If $\p \neq \m$, then as noted in the proof of \Cref{J a saturation}, we have $I_\p = (f,g)_\p$, hence $\mu(I_\p) + \dim R/\p \leq n$. Lastly, if $\p=\m$, then it is clear that $\mu(I_\p) + \dim R/\p = n+1$. With all cases considered, it follows that $\dim \S(I)=n+1$ as claimed. The second statement then follows immediately.
\end{proof}

As noted in \Cref{S(I) remark}, we have that $\S(I) \cong R[y_1,\ldots,y_{n+1}]/\L$ where $\L = I_2(\psi) + (h)$. We adopt the bigrading on $R[y_1,\ldots,y_{n+1}]$ given by $\bideg x_i=(1,0)$ and $\bideg y_i = (0,1)$ throughout. With this, we now produce a minimal bigraded free resolution of $\S(I)$.

\begin{prop}\label{resolution prop}
Writing $B=R[y_1,\ldots,y_{n+1}]$, the symmetric algebra $\S(I)$ has a minimal bigraded free $B$-resolution
$$F_\dit:\,0\rightarrow F_n \rightarrow \cdots\rightarrow F_1\rightarrow F_0$$
where 
$$F_n = B^{n-1} (-d-n+2,-n), \quad  F_i = \begin{array}{c}
B^{\binom{n}{i}(i-1)}(-d-i+2,-i)\\
\oplus \\
B^{\binom{n}{i+1}i}(-i,-i)
\end{array},  \quad  F_1 = \begin{array}{c}
B(-d+1,-1)\\
\oplus \\
B^{\binom{n}{2}}(-1,-1)
\end{array}, \quad F_0 =B $$
for $2\leq i\leq n-1$. In particular, $\S(I)$ is Cohen-Macaulay.
\end{prop}

\begin{proof}
First note that since $\psi$ is a matrix of indeterminates, it is well known that $I_2(\psi)$ is a prime ideal of height $n-1$. Since $\hgt \L =n$ by \Cref{S(I) dimension}, it follows that $h\notin I_2(\psi)$, and so $h$ is regular modulo this ideal of minors. As such, there is a bigraded short exact sequence
\begin{equation}\label{mapping cone ses}
0\rightarrow B/I_2(\psi) (-d+1,-1) \overset{\cdot h}{\longrightarrow} B/I_2(\psi) \rightarrow \S(I)\rightarrow 0.
\end{equation}

As $I_2(\psi)$ has generic height, we have that $B/I_2(\psi)$ is resolved by the Eagon-Northcott complex \cite[A2.10]{Eisenbud}. Moreover, multiplication by $h$ lifts to a morphism of complexes from the Eagon-Northcott complex to itself. With this, applying the mapping cone construction within (\ref{mapping cone ses}) gives the bigraded free resolution of $\S(I)$ above. Moreover, from the bidegree shifts involved, it follows that this is a minimal free resolution. Lastly, since $\hgt \L =n$ by \Cref{S(I) dimension}, it follows that $\S(I)$ is Cohen-Macaulay by the Auslander-Buchsbaum formula \cite[19.9]{Eisenbud}.
\end{proof}

We now describe the saturation $\J= \L:\m^\infty$ as a colon ideal for a particular power of $\m$. We make use of the tools introduced in \cite{KPU20} to bound the generation degree of certain local cohomology modules. First however, we must briefly recall some notation and terminology of \cite{KPU20}. For $S$ a nonnegatively graded ring with $S_0$ local and $M$ an $S$-module, write $b_0(M)=\inf\{\, j \,|\, S(\bigoplus_{i\leq j} M_i) =M\}$ to denote the \textit{maximal generator degree} of $M$, and $a(M) = \sup\{\, i\,|\, [H_{\mathfrak{M}}^{\dim S}(M)]_i \neq 0\}$ to denote the \textit{a-invariant} of $M$, where $\mathfrak{M}$ denotes the homogeneous maximal ideal of $S$.

\begin{prop}\label{J a colon ideal}
  With the assumptions of \Cref{main setting}, we have $\J = \L:\m^{d-1}$.
\end{prop}

\begin{proof}
Recall that there is an induced epimorphism $\S(I)\rightarrow \R(I)$ and notice that, following \Cref{J a saturation}, the kernel of this map is precisely the local cohomology module $\A=H_\m^0(\S(I))$. Moreover, notice that $\A$ is naturally bigraded, where $\bideg x_i =(1,0)$ and $\bideg y_i =(0,1)$ as before. Hence for any fixed $q$,
we have
$$\A_{(*,q)} = \bigoplus_p \A_{(p,q)} \cong H_\m^0(\S_q(I)).$$
As $\A$ lives in finitely many degrees, we claim that $\A$ vanishes past degree $d-2$ in the first component of the bigrading, i.e. that $\A_{(p,q)} = 0$ for any $q$ and $p> d-2$.

With the $B$-resolution $F_\dit$ of $\S(I)$ given in \Cref{resolution prop}, taking the graded strand in degree $(*,q)$ yields a graded $R$-resolution of $\S_q(I)$. With this and \cite[3.8]{KPU20}, it follows that $\A_{(p,q)} =0$ for all $p>b_0([F_n]_{(*,q)}) + a(R)$, where  $[F_n]_{(*,q)}$ is the free $R$-module $[F_n]_{(*,q)} = \bigoplus_p [F_n]_{(p,q)}$. From \Cref{resolution prop} it follows that $b_0([F_n]_{(*,q)}) \leq d+n-2$. Additionally, since $R$ is a standard graded polynomial ring in $n$ variables, we have that $a(R) =-n$ \cite[3.6.15]{BH93}. Hence $\A_{(p,q)} =0$ for any $q$ and $p>d-2$, as claimed. Thus it follows that $\m^{d-1} \A=0$ and so $\m^{d-1} \J \subseteq \L$. Thus we have $\J\subseteq \L:\m^{d-1}$, and the reverse containment follows from \Cref{J a saturation}.
\end{proof}


\section{Approximations of Rees algebras}\label{approximation section}

In this section, we approximate the Rees ring $\R(I)$, in a sense, by employing the so-called \textit{method of divisors}, as introduced in \cite{KPU11}. We map onto $\R(I)$ with a closely related algebra, and study the kernel of this map. We then aim to lift a generating set of this kernel to generators of $\J$.

Consider the subideal $f\m= (fx_1,\ldots,fx_n)$, and notice that we may factor the natural epimorphism $R[y_1,\ldots,y_{n+1}]\rightarrow \R(I)=R[fx_1t,\ldots,fx_nt,gt]$, as in the introduction, through the Rees ring of this ideal. Indeed, we have the following commutative diagram
\begin{equation}\label{R(m) -> R(I) diagram}
    \SelectTips{cm}{}
    \xymatrix{
    R[y_1,\ldots,y_{n+1}] \ar[dr] \ar[rr]  & & \R(I)\\
     & \R(f\m)[y_{n+1}] \ar[ur] & }
\end{equation}
where the horizontal map is the natural map above. Moreover, the leftmost map sends $y_i\mapsto fx_it$ for $1\leq i\leq n$ and $y_{n+1} \mapsto y_{n+1}$, and the rightmost map then sends $y_{n+1}\mapsto gt$.

\begin{rem}\label{R(fm) and R(m) isomorphic remark}
Notice that there is an isomorphism of Rees algebras $\R(f\m)\cong \R(\m)$. This follows as $f$ is a non-zerodivisor, hence $f\m$ and $\m$ have the same syzygies, but this isomorphism can also be made explicit. As $\R(\m) = R[x_1s,\ldots,x_ns]\subseteq R[s]$ and $\R(f\m) = R[fx_1t,\ldots,fx_nt]\subset R[t]$, for $s,t$ indeterminates, the isomorphism $\R(\m)\cong \R(f\m)$ is obtained as a restriction of the map $R[s]\rightarrow R[t]$ given by $s\mapsto ft$, again noting that $f$ is a non-zerodivisor. In particular, we see that $\R(f\m)\cong \R(\m) \cong R[y_1,\ldots,y_n]/I_2(\psi)$ where
\[
\psi =\begin{bmatrix}
    x_1&\ldots &x_n\\
    y_1&\ldots& y_n
\end{bmatrix}
\]
as before in (\ref{psi defn}), following \cite[Chap. I: Thm. 1, Lem. 2]{Micali64}.
\end{rem}

We note that the induced map of (\ref{R(m) -> R(I) diagram}) may also be realized using techniques from Rees algebras of \textit{modules}, and similar ideas used in \cite{CPW23}. With $\varphi$ as in \Cref{Presentation of I}, write $\varphi'$ to denote the submatrix obtained by deleting its last column. There is an epimorphism $\coker \varphi' \rightarrow \coker \varphi$, noting that $\coker \varphi' \cong \m\oplus R$ and certainly $\coker \varphi = I$. Hence there is an induced map of Rees algebras $\R(\m\oplus R) \cong \R(\m)[y_{n+1}] \rightarrow \R(I)$, precisely as above.

\begin{defn}
For simpler notation, write $A = \R(f\m)[y_{n+1}]$, noting that $A\cong R[y_1,\ldots,y_{n+1}]/I_2(\psi)$ by \Cref{R(fm) and R(m) isomorphic remark}. With this, write $\overline{\,\cdot\,}$ to denote images modulo $I_2(\psi)$ in the ring $A$. Lastly, we define the $R[y_1,\ldots,y_{n+1}]$-ideal $\K = I_2(\psi)+(x_n,y_n)$.
\end{defn}

With the above conventions, we note that there is a similarity between the present situation and that of \cite{Weaver25}. For the duration of this section, we proceed along a very similar path and arrive at analogous results.

\begin{prop}\label{K and m properties}
The ring $A$ is a Cohen-Macaulay domain of dimension $n+2$. Additionally, the $A$-ideals $\overline{\m}$ and $\overline{\K}$ are Cohen-Macaulay of height one, and $\overline{\m}$ is a prime ideal.
\end{prop}

\begin{proof}
It is clear that $A$ is a domain of dimension $n+2$, since $R$ is a domain of dimension $n$. Moreover, $A$ is easily seen to be Cohen-Macaulay as it is defined by a determinantal ideal of maximal height \cite[A2.13]{Eisenbud}. For the assertions on $\overline{\m}$, notice that $A/\overline{\m} \cong \F(\m)[y_{n+1}]$ where $\F(\m)$ denotes the special fiber ring of $\m$. From the description of $\R(\m)$ in \Cref{R(fm) and R(m) isomorphic remark}, it follows that $\F(\m)\cong k[y_1,\ldots,y_n]$, hence $\overline{\m}$ is a Cohen-Macaulay prime ideal of height one. Lastly, for the assertion on $\overline{\K}$, notice that
$$\K = I_2(\psi) +(x_n,y_n) =I_2\left(\begin{bmatrix}
    x_1&\cdots &x_{n-1}\\
    y_1&\cdots& y_{n-1}
\end{bmatrix}\right) +(x_n,y_n).$$
From this description and \cite[A2.13]{Eisenbud}, it follows that $\K$ is Cohen-Macaulay of height $(n-2) +2 =n$. Hence modulo $I_2(\psi)$, we have that $\overline{\K}$ is Cohen-Macaulay of height one.
\end{proof}

With the properties of $\overline{\m}$ and $\overline{\K}$ in \Cref{K and m properties}, we now show that the symbolic powers of these $A$-ideals are \textit{linked} \cite{Huneke82}.

\begin{prop}\label{linkage prop}
   With $\overline{\K}$ and $\overline{\m}$ as above, we have the following.
\begin{enumerate}
    \item[(a)] $\overline{\m}^i = \overline{\m}^{(i)}$, 

    \item[(b)] $(\overline{x_n}^i):\overline{\K}^{(i)} = \overline{\m}^{(i)}$,

    \item[(c)] $(\overline{x_n}^i):\overline{\m}^{(i)} = \overline{\K}^{(i)}$,

\end{enumerate}
for all $i\in \mathbb{N}$. 
\end{prop}

\begin{proof}
We proceed as in the proof of \cite[3.9]{BM16} (see also \cite[3.3]{Weaver25}).

\begin{enumerate}
\item[(a)] Temporarily regrading $R[y_1,\ldots,y_{n+1}]$ by setting $\deg x_i= 1$ and $\deg y_i=0$, we note that $\G(\overline{\m}) \cong A$, where $\G(\overline{\m})$ denotes the associated graded ring of $\overline{\m}$. Recall from \Cref{K and m properties} that $A$ is a domain, hence the claim follows.  

\vspace{1mm}

\item[(b)] Notice that modulo $I_2(\psi)$ we have $\overline{x_iy_n} = \overline{x_ny_i}$ for $1\leq i\leq n$, hence it is clear that $\overline{\K} \overline{\m} \subseteq (\overline{x_n})$, and so $\overline{\K}^i\overline{\m}^i \subseteq (\overline{x_n}^i)$ for any $i$. Localizing at primes of height one and contracting, we have $\overline{\K}^{(i)}\overline{\m}^{(i)} \subseteq (\overline{x_n}^i)$. Thus $\overline{\m}^{(i)} \subseteq (\overline{x_n}^i):\overline{\K}^{(i)}$, and so we need only show the reverse containment. Since $\overline{y_n}\notin \overline{\m}$, we note that $\overline{\K} \nsubseteq \overline{\m}$. As $\overline{\m}$ is the unique associated prime of $\overline{\m}^{(i)}$, it then follows that $\overline{\K}^{(i)}$ and $\overline{\m}^{(i)}$ have no common associated prime. With this, it follows that $(\overline{x_n}^i):\overline{\K}^{(i)}\subseteq \overline{\m}^{(i)}$ as well, and so $(\overline{x_n}^i):\overline{\K}^{(i)} = \overline{\m}^{(i)}$.

\vspace{1mm}

\item[(c)] As previously noted, we have $\overline{\K}^{(i)}\overline{\m}^{(i)} \subseteq (\overline{x_n}^i)$, hence there is a containment $\overline{\K}^{(i)} \subseteq (\overline{x_n}^i):\overline{\m}^{(i)}$. Again noting that $\overline{\K}^{(i)}$ and $\overline{\m}^{(i)}$ share no common associated prime, it follows that $(\overline{x_n}^i):\overline{\m}^{(i)} \subseteq \overline{\K}^{(i)}$, and so $(\overline{x_n}^i):\overline{\m}^{(i)} = \overline{\K}^{(i)}$.\qedhere
\end{enumerate}
\end{proof}

\begin{cor}\label{Kbar SCM and generically a CI}
The $A$-ideal $\overline{\K}$ is generically a complete intersection and is a strongly Cohen-Macaulay ideal. 
\end{cor}

\begin{proof}
Recall from the proof of \Cref{linkage prop} that $\overline{\K}$ and $\overline{\m}$ have no common associated prime. Thus for any associated prime $\p$ of $\overline{\K}$, by \Cref{linkage prop} we have $\overline{\K}_\p  =(\overline{x_n})_\p: A_\p =(\overline{x_n})_\p$, which shows that $\overline{\K}$ is generically a complete intersection. Furthermore, by \Cref{K and m properties} and noting that $\overline{\K} = \overline{(x_n,y_n)}$, we see that $\overline{\K}$ is a Cohen-Macaulay almost complete intersection ideal of height one. Since it has just been shown to be generically a complete intersection, it then follows from \cite[2.2]{Huneke83} that $\overline{\K}$ is strongly Cohen-Macaulay.
\end{proof}

We now offer a description of the $A$-ideal $\overline{\J}$, noting that this is the kernel of the induced map $A\rightarrow \R(I)$ in (\ref{R(m) -> R(I) diagram}). Recall from \Cref{S(I) remark} that $\L= I_2(\psi)+(h)$ is the defining ideal of $\S(I)$, and that $h$ is a bihomogeneous polynomial of bidegree $(d-1,1)$. With this, consider the divisorial ideal $\D = \frac{\overline{h}\overline{\K}^{(d-1)} }{\overline{x_n}^{d-1}}$, and note that this is an $A$-ideal by \Cref{linkage prop}, since $h \in \m^{d-1}$.

\begin{thm}\label{Jbar=D}
   With the assumptions of \Cref{main setting}, we have that $\overline{\J} = \D.$
\end{thm}

\begin{proof}
We first show the containment $\D \subseteq \overline{\J}$. By \Cref{J a colon ideal}, we have $\J = \L:\m^{d-1}$, and so we must show that $\D \cdot \overline{a} \subset \overline{\L} = (\overline{h})$ for any $a\in \m^{d-1}$. This follows as
$$\D \cdot \overline{a} = \frac{\overline{h}\overline{\K}^{(d-1)} }{\overline{x_n}^{d-1}} \cdot \overline{a} = \frac{\overline{a}\overline{\K}^{(d-1)} }{\overline{x_n}^{d-1}} \cdot \overline{h} \subseteq (\overline{h}),$$
where the last containment follows as $\tfrac{\overline{a}\overline{\K}^{(d-1)} }{\overline{x_n}^{d-1}}$ is an $A$-ideal, by \Cref{linkage prop}.

With the containment $\D \subseteq \overline{\J}$, it suffices to show equality locally at the associated primes of $\D$ in order to conclude that $\overline{\J} = \D$. Notice that there is an isomorphism $\D \cong \overline{\K}^{(d-1)}$, where the latter ideal is unmixed of height one by \Cref{linkage prop} and \cite[0.1]{Huneke82}. As such, $\D$ is also unmixed of height one, and so it suffices to show that $\D_\p = \overline{\J}_\p$ for any prime ideal $\p$ of $A$ with $\hgt \p =1$.

Recall from \Cref{K and m properties} that $\overline{\m}$ is a prime $A$-ideal of height one. If $\p \neq \overline{\m}$, then from \Cref{linkage prop} we have $\overline{\K}^{(d-1)}_\p  =(\overline{x_n}^{d-1})_\p:A_\p = (\overline{x_n}^{d-1})_\p$, hence $\D_\p = (\overline{h})_\p$. Similarly, by \Cref{J a colon ideal} we have $\overline{\J}_\p =  (\overline{h})_\p:A_\p = (\overline{h})_\p$ as well. Hence $\D_\p = \overline{\J}_\p$ for any height one prime ideal $\p \neq \overline{\m}$, and so we need only consider the case that $\p=\overline{\m}$.
We first note that $\overline{\J} \nsubseteq \overline{\m}$, and so locally we have $\overline{\J}_{\overline{\m}} = A_{\overline{\m}}$. Indeed, from the isomorphism $\R(I) \cong A/\overline{\J}$, it follows that the fiber ring is $\F(I)\cong A/(\overline{\J}+\overline{\m})$ which has dimension at most $\dim R =n$. However, $A/\overline{\m} \cong k[y_1,\ldots,y_{n+1}]$, and so we have $\overline{\J} \nsubseteq \overline{\m}$.
With this, we need only show that $\D_{\overline{\m}}$ is the unit ideal as well.

Recall from the proof of \Cref{linkage prop} that $\overline{\K} \nsubseteq \overline{\m}$, hence $\overline{\K}^{(d-1)}_{\overline{\m}} = A_{\overline{\m}}$. Thus by \Cref{linkage prop} we see that $(\overline{x_n}^{d-1})_{\overline{\m}}:\overline{\m}^{d-1}_{\overline{\m}} = A_{\overline{\m}}$, and so $\overline{\m}^{d-1}_{\overline{\m}} \subseteq (\overline{x_n}^{d-1})_{\overline{\m}}$. As such, we have $\overline{\m}^{d-1}_{\overline{\m}} = (\overline{x_n}^{d-1})_{\overline{\m}}$, as the reverse containment is clear. Similarly, since $\overline{\J}_{\overline{\m}}= A_{\overline{\m}}$, by \Cref{J a colon ideal} we have $(\overline{h})_{\overline{\m}}:\overline{\m}^{d-1}_{\overline{\m}} = A_{\overline{\m}}$. Thus $\overline{\m}^{d-1}_{\overline{\m}} \subseteq (\overline{h})_{\overline{\m}}$, and since $h \in \m^{d-1}$, we have the equality $\overline{\m}^{d-1}_{\overline{\m}} = (\overline{h})_{\overline{\m}}$. With this, we have $(\overline{x_n}^{d-1})_{\overline{\m}}=\overline{\m}^{d-1}_{\overline{\m}} = (\overline{h})_{\overline{\m}}$, hence $\D_{\overline{\m}} = \overline{\K}^{(d-1)}_{\overline{\m}} = A_{\overline{\m}}$.
\end{proof}


\section{Downgraded sequences}\label{Algorithm Section}

In this section, we present an adaptation of the method of \textit{downgraded sequences} given by Hassanzadeh and Simis in \cite{HS14}. Whereas the relation between such a sequence and the defining ideal has been studied in greater generality \cite[\textsection 4]{HS14}, we present this algorithm from the specific viewpoint when the underlying Cremona support is the identity. With this, our approach throughout is remarkably similar to that of \cite[\textsection 4]{Weaver25} and this particular perspective will be of use in the following section, when it is shown that the defining ideal coincides with the ideal of a full downgraded sequence. Before we proceed with the construction of downgraded sequences, we recall some previous notation from \Cref{Main Section}.

\begin{notat}\label{del notation}
For any bihomogeneous polynomial $p\in R[y_1,\ldots,y_{n+1}]$, write $\partial p$ to denote a column of bihomogeneous elements, of the same bidegree, such that 
$$[x_1\ldots x_n] \cdot \partial p = p.$$
As a convention, we take $\partial p =0$ if $p=0$.
\end{notat}

As noted in the observation following \Cref{Presentation of I}, there is generally not a unique choice for $\partial p$. Despite this non-uniqueness, any choice will suffice our purpose. Moreover, the ideals obtained from the following process will be shown to be independent of such a choice.

\begin{defn}\label{algorithm defn}
With the notation above and the assumptions of \Cref{main setting}, we recursively define a sequence of polynomials and ideals. For the initial pair, set $h_1=h$ and $\J_1 = \L$, as in \Cref{S(I) remark}. For $1\leq i\leq d$, suppose that $(\J_1,h_1), \ldots, (\J_{i-1},h_{i-1})$ have been constructed. To construct the $i$th pair, set $h_i = [y_1\ldots y_n]\cdot \partial h_{i-1}$ and let $\J_i = \J_{i-1} +(h_i)$. The sequence $h_1,\ldots,h_i$ is called an $i$th \textit{downgraded sequence} of $h$.
\end{defn}

As noted, there are multiple choices for $\partial h_i$ at each step, hence a downgraded sequence is not unique in general. However, we will show that each ideal $\J_i = I_2(\psi) +(h_1,\ldots,h_i)$ is well defined. First however, we make a brief observation about the polynomials of a downgraded sequence.

\begin{prop}\label{hi bideg remark}
For $h_1,\ldots,h_d$ a downgraded sequence as above, we have $h_i\neq 0$ for all $1\leq i\leq d$. Moreover, each $h_i$ is bihomogeneous of bidegree $(d-i,i)$.
\end{prop}

\begin{proof}
As the process of \Cref{algorithm defn} is recursive, we prove the first statement by induction. The second statement then follows immediately, by construction. Recall that $h=h_1$ is a nonzero polynomial of bidegree $(d-1,1)$, as noted in \Cref{S(I) remark} and from the presentation of $I$ in \Cref{Presentation of I}.

Now suppose, for a contradiction, that $h_i=0$ for some $i\geq 2$. As such, we may take $i$ to be minimal, in the sense that $h_j\neq 0$ for $j<i$. Since $h_i=[y_1\ldots y_n]\cdot \partial h_{i-1}$ and $h_i=0$, we see that $\partial h_{i-1}$ is a syzygy on $y_1,\ldots,y_n$. As this is a regular sequence, $\partial h_{i-1}$ belongs to the span of its Koszul syzygies. With this, and since $h_{i-1}=[x_1\ldots x_n]\cdot \partial h_{i-1}$, it then follows that $h_{i-1} \in I_2(\psi)$.

If $i=2$, this is a contradiction as it was shown in the proof of \Cref{S(I) dimension} that $h_1=h\notin I_2(\psi)$, hence we must have $i\geq 3$ here. Thus we have $h_{i-1} = [y_1\ldots y_n] \cdot \partial h_{i-2}$ and, since $h_{i-1}\in I_2(\psi)$, it follows that $\partial h_{i-2}$ is contained in the span of Koszul syzygies on $x_1,\ldots, x_n$. However, then $h_{i-2} = [x_1\ldots x_n]\cdot \partial h_{i-2} =0$, which is a contradiction as $i$ was taken to be smallest.
\end{proof}

\begin{prop}\label{Ji well defined}
The ideals $\J_1,\ldots,\J_d$ of \Cref{algorithm defn} are well defined.
\end{prop}

\begin{proof}
We proceed by induction, noting that the initial ideal $\J_1 = \L= I_2(\psi)+(h_1)$ is well defined since this is the ideal defining $\S(I)$, as noted in \Cref{S(I) remark}. Now suppose that $2\leq i \leq d$ and the ideals $\J_1,\ldots,\J_{i-1}$ have been shown to be unique. Letting $\partial h_{i-1}$ and $\partial h_{i-1}'$ denote two columns with 
$$[x_1\ldots x_n]\cdot \partial h_{i-1} = h_{i-1} = [x_1\ldots x_n]\cdot \partial h_{i-1}'$$
as in \Cref{algorithm defn}, we must show that $\J_{i-1}+(h_i) = \J_{i-1}+(h_i')$, where $h_i=[y_1\ldots y_n]\cdot \partial h_{i-1}$ and $h_i'=[y_1\ldots y_n]\cdot \partial h_{i-1}'$.

From the expression above, we see that $\partial h_{i-1} - \partial h_{i-1}'$ is a syzygy on $x_1,\ldots,x_n$. Since this is a regular sequence, $\partial h_{i-1}-\partial h_{i-1}'$ belongs to the span of its Koszul syzygies. As such, it follows that $h_i=[y_1\ldots y_n] \cdot \partial h_{i-1}$ and $h_i'=[y_1\ldots y_n] \cdot  \partial h_{i-1}'$ differ by an element of $I_2(\psi)$. Now since $I_2(\psi) \subseteq \J_{i-1}$, we have that $\J_{i-1}+(h_i) = \J_{i-1}+(h_i')$, and the claim follows.
\end{proof}

We note that the process of downgraded sequences is relatively simple in practice, and this algorithm may be easily implemented into a computer algebra system, such as \textit{Macaulay2} \cite{Macaulay2}.

\begin{ex}
Let $R=k[x_1,x_2,x_3]$ with homogeneous maximal ideal $\m=(x_1,x_2,x_3)$, and let $f=x_1^2$ and $g=x_2^3$. Consider the ideal $I=(f\m,g)=(x_1^3,x_1^2x_2,x_1^2x_3,x_2^3)$, which satisfies the assumptions of \Cref{main setting}. Following the construction in \Cref{Presentation of I}, one has that $h=x_2^2y_2-x_1^2y_4$. With this, by \Cref{algorithm defn} a downgraded sequence of $h$ may be taken as
 \[
    \begin{array}{ccccc}
  h_1=x_2^2y_2-x_1^2y_4, &\quad &  h_2=x_2y_2^2-x_1y_1y_4, &\quad  & h_3=y_2^2-y_1^2y_4. 
\end{array}
\]
\end{ex}

We note that the algorithm outlined in \Cref{algorithm defn} may be taken as an exchange process. Namely, within each monomial in the support of $h_i$, one may exchange exactly one of the $x_j$ for $y_j$ in order to produce $h_{i+1}$. One then repeats this until reaching $h_d$, when there are no $x_j$ left to exchange, following \Cref{hi bideg remark}. Whereas this is not the only way to produce a downgraded sequence in general, \Cref{Ji well defined} guarantees the ideals obtained from this method agrees with the ideals obtained from any other choice.

\begin{prop}\label{Ji contained in colons}
  For $1\leq i\leq d$, we have the containment $\J_i\subseteq \L:\m^{i-1}$.
\end{prop}

\begin{proof}
We proceed by induction, noting that the assertion is clear when $i=1$, since $\J_1=\L$. Now suppose that $2\leq i\leq d$ and $\J_{i-1}=I_2(\psi)+(h_1,\ldots,h_{i-1})\subseteq \L:\m^{i-2}$. We claim that $h_i\m +I_2(\psi) = h_{i-1}(y_1,\ldots,y_n)+I_2(\psi)$, for $h_i$ the next downgraded polynomial. Indeed, modulo $I_2(\psi)$, for any $j$ we have
\begin{equation}\label{hi in colon}
\overline{x_jh_i} = [\overline{x_jy_1}\ldots \overline{x_j y_n}] \cdot \overline{\partial h_{i-1}} = [\overline{x_1y_j}\ldots \overline{x_n y_j}] \cdot \overline{\partial h_{i-1}} = \overline{y_jh_{i-1}},
\end{equation}
from which the claim follows. With this, we have that $\J_i = \J_{i-1} +(h_i) \subseteq \J_{i-1}:\m \subseteq \L:\m^{i-1}$, where the latter containment follows from the induction hypothesis.
\end{proof}

Notice that \Cref{Ji contained in colons} and \Cref{J a colon ideal} show that $\J_{d} \subseteq \J$, where $\J_d=I_2(\psi)+(h_1,\ldots,h_{d})$ is the ideal of the full downgraded sequence. We will eventually show that this containment is an equality. First however, we produce a description of $\overline{\J_d}$, similar to the one for $\overline{\J}$ in \Cref{Jbar=D}.

\begin{thm}\label{Jd divisor}
With the assumptions of \Cref{main setting} and $\J_{d}$ the ideal of the full downgraded sequence of \Cref{algorithm defn}, we have $\overline{\J_{d}} = \frac{\overline{h}\overline{\K}^{d-1} }{\overline{x_n}^{d-1}}$.
\end{thm}

\begin{proof}
We note first that $\frac{\overline{h}\overline{\K}^{d-1} }{\overline{x_n}^{d-1}}$ is an $A$-ideal by \Cref{linkage prop}. With this and the definition of $\K$, we must show that $\overline{h_1} \overline{(x_n,y_n)}^{d-1} =\overline{x_n}^{d-1}\overline{(h_1,\ldots,h_{d})}$. In particular, it is enough to show that $\overline{h_1} \overline{x_n^{d-i}y_n^{i-1}} = \overline{x_n}^{d-1}\overline{h_i}$ or rather, recalling that $A$ is a domain, that $\overline{h_1} \overline{y_n}^{i-1} = \overline{x_n}^{i-1}\overline{h_i}$ for $1\leq i\leq d$. We show this by induction and note that there is nothing to be shown for the case $i=1$. Hence we consider the case $i=2$ initially, and begin by producing an expression similar to (\ref{hi in colon}). Recall that $h_1 = [x_1\ldots x_n]\cdot \partial h_1$ and $h_2 = [y_1\ldots y_n]\cdot \partial h_1$. Thus modulo $I_2(\psi)$, we have 
    $$\overline{h_1 y_n}  = [\overline{y_n x_1}\ldots \overline{y_n x_n}]\cdot \overline{\partial h_1} = [\overline{y_1 x_n}\ldots \overline{y_n x_n}]\cdot \overline{\partial h_1}= \overline{x_n h_2}$$
    and the initial claim follows.

We are finished if $d=2$, so suppose that $3\leq i\leq d$ and the claim holds up to $i-1$. To show that $\overline{h_1} \overline{y_n}^{i-1} = \overline{x_n}^{i-1}\overline{h_i}$, we note that 
$$\overline{h_1} \overline{y_n}^{i-1} = \overline{h_1}\overline{y_n}^{i-2} \cdot \overline{y_n} = \overline{x_n}^{i-2}\overline{h_{i-1}}\cdot \overline{y_n} = \overline{x_n}^{i-2}\cdot\overline{h_{i-1}} \overline{y_n},$$
by applying the induction hypothesis. Hence the proof will be complete once it has been shown that $\overline{h_{i-1}}\overline{y_n} = \overline{x_n}\overline{h_i}$. However, this follows exactly as above. Since
$h_{i-1} = [x_1\ldots x_n]\cdot \partial h_{i-1}$ and $h_i = [y_1\ldots y_n]\cdot \partial h_{i-1}$, modulo $I_2(\psi)$ we have 
$$\overline{h_{i-1} y_n}  = [\overline{y_n x_1}\ldots \overline{y_n x_n}]\cdot \overline{\partial h_{i-1}} = [\overline{y_1 x_n}\ldots \overline{y_n x_n}]\cdot \overline{\partial h_{i-1}}= \overline{x_n h_i}$$
and the claim follows.
\end{proof}


\section{Implicitization and Cohen-Macaulayness}\label{Defining Ideal Section}

We now show that the defining ideal $\J$ of the Rees ring $\R(I)$ coincides with the ideal $\J_d$ of the fully downgraded sequence of \Cref{Algorithm Section}. The remainder of the section is then devoted to the Cohen-Macaulay property of $\R(I)$.

\begin{thm}\label{defining ideal theorem}
  With the assumptions of \Cref{main setting}, the defining ideal of $\R(I)$ is $\J = I_2(\psi) +(h_1,\ldots,h_d)$.
\end{thm}

\begin{proof}
From the containments $I_2(\psi) \subseteq \J_d \subseteq \J$, it suffices to show that $\overline{\J_d} = \overline{\J}$. With this, by \Cref{Jbar=D} and \Cref{Jd divisor}, it is then enough to show that  $\overline{\K}^{d-1} = \overline{\K}^{(d-1)}$. Recall from \Cref{Kbar SCM and generically a CI} that $\overline{\K}$ is generically a complete intersection and is a strongly Cohen-Macaulay ideal of height one. Thus by \cite[3.4]{SV81}, it is enough to show that $\overline{\K}_\p$ is a principal $A_\p$-ideal for any prime $A$-ideal $\p$ with $\hgt \p =2$, as then $\overline{\K}^{i} = \overline{\K}^{(i)}$ for all $i\geq 1$.

 First note that if $\p \nsupseteq \overline{\m}$, we have $\overline{\K}_\p = (\overline{x_n})_\p:A_\p = (\overline{x_n})_\p$ by \Cref{linkage prop}. Hence we may assume that $\p$ is a prime $A$-ideal of height two with $\p \supseteq \overline{\m} =\overline{(x_1,\ldots,x_n)}$. In this case, there must exist a $\overline{y_i}$ for $1\leq i\leq n$ such that $\overline{y_i} \notin \p$. If not, then $\overline{(x_1,\ldots, x_n,y_1,\ldots,y_n)} \subseteq \p$, which is impossible as $\hgt \p =2$ and the former ideal has height $2n-(n-1) = n+1 \geq 3$ (as $n\geq \hgt I=2$). Now since $\overline{y_i}\notin \p$ becomes a unit locally in $A_\p$, the equation $\overline{x_iy_n} = \overline{x_ny_i}$ shows that $\overline{\K}_\p = \overline{(x_n,y_n)}_\p = (\overline{y_n})_\p$, which completes the proof.
\end{proof}

\begin{cor}\label{fiber ring cor}
With the assumptions of \Cref{main setting}, the fiber ring of $I$ is $\F(I) \cong k[y_1,\ldots,y_{n+1}]/(h_d)$. In particular, $\F(I)$ is Cohen-Macaulay.
 \end{cor}

\begin{proof}
From \Cref{defining ideal theorem}, the defining ideal of $\R(I)$ is $\J = I_2(\psi) +(h_1,\ldots,h_d)$. We note that $I_2(\psi)$ is generated in bidegree $(1,1)$ and each $h_i$ is nonzero of bidegree $(d-i,i)$ by \Cref{hi bideg remark}. Hence from bidegree considerations, we have $\J + \m =(h_d) +\m$, and so $\F(I) \cong R[y_1,\ldots,y_{n+1}]/(\J+\m) \cong k[y_1,\ldots,y_{n+1}]/(h_d)$. The Cohen-Macaulayness is clear since $\F(I)$ is a hypersurface ring.
\end{proof}

\begin{rem}
As noted, the de Jonqui\`eres transformation $\deJonq:\P^{n-1} \dashrightarrow \P^n$ given by $fx_1,\ldots,fx_n,g$ is birational onto its image, which is a hypersurface in $\P^n$. Following \Cref{fiber ring cor}, this hypersurface is $V(h_d)$.
\end{rem}

We end this section by describing the Cohen-Macaulayness of $\R(I)$. First however, we make a brief observation about the $A$-ideal $\overline{\K}$ and its symmetric powers.

\begin{lemma}\label{Kbar linear type}
The $A$-ideal $\overline{\K}$ is of linear type. In particular, $\S_i(\overline{\K}) \cong \overline{\K}^i$ for all $i$.
\end{lemma}

\begin{proof}
 Recall that $\overline{\K}$ is a Cohen-Macaulay ideal of height one and generically a complete intersection by \Cref{K and m properties} and \Cref{Kbar SCM and generically a CI}. With this, and noting that $\overline{\K} = \overline{(x_n,y_n)}$, it follows that $\mu(\overline{\K}_\p)\leq \dim A_\p$ for all $\p\in \spec(A)$, where $\mu(-)$ denotes minimal number of generators. Furthermore, since $\overline{\K}$ is strongly Cohen-Macaulay by \Cref{Kbar SCM and generically a CI}, it follows that $\overline{\K}$ is of linear type by \cite[2.6]{HSV82}.
\end{proof}

With \Cref{Kbar linear type}, we may now address the Cohen-Macaulayness of $\R(I)$. From \Cref{Jd divisor} and the proof of \Cref{defining ideal theorem}, we have an isomorphism $\overline{\J} \cong \overline{\K}^{d-1}$, hence the Cohen-Macaulayness of $\R(I)$ is determined by when $\overline{\K}^{d-1}$ is a Cohen-Macaulay ideal. Recall that a Noetherian local ring $S$ is said to be \textit{almost} Cohen-Macaulay if $\dim S -\depth S \leq 1$.

\begin{thm}\label{Cohen-Macaulayness of Rees algebra}
With the assumptions of \Cref{main setting}, $\R(I)$ is almost Cohen-Macaulay, and is Cohen-Macaulay if and only if $d\leq n$. 
\end{thm}

\begin{proof}
Recall that $\K = I_2(\psi) +(x_n,y_n)$ for $\psi$ as in (\ref{psi defn}). Thus by \cite[A2.14]{Eisenbud}, there is an isomorphism $\overline{\K} \cong M$, where $M=\coker \psi$. Since $\psi$ consists of indeterminate entries, the symmetric powers $\S_i(M)$ are resolved by the family of Eagon-Northcott complexes \cite[A2.10]{Eisenbud}. Moreover, by \Cref{Kbar linear type} we have that $\overline{\K}^i \cong \S_i(\overline{\K}) \cong \S_i(M)$ for all $i\geq 0$.

With this, if $i\leq n-1$, then $\overline{\K}^i \cong \S_i(M)$ and has a minimal free resolution of length $n-1$, and is therefore a maximal Cohen-Macaulay $A$-module \cite[A2.13, A2.14a]{Eisenbud}. Moreover, if $i\geq n$, then $\overline{\K}^i \cong \S_i(M)$ has a minimal free resolution of length $n$ \cite[A2.10, A2.14b]{Eisenbud}. Thus from the Auslander-Buchsbaum formula \cite[19.9]{Eisenbud}, the depth of each power, as a module, is 
\begin{equation}\label{depth Kbar^i}
\depth \overline{\K}^i = \left\{
     \begin{array}{ll}
       n+2 & \text{if $0\leq i\leq n-1$,}\\[1ex]
       n+1 & \text{if $i\geq n$.}
     \end{array}
   \right.
\end{equation}

With (\ref{depth Kbar^i}) and the isomorphism $\overline{\J} \cong \overline{\K}^{d-1}$, 
comparing depths along the short exact sequence
$$0\rightarrow \overline{\J} \rightarrow A\rightarrow \R(I) \rightarrow 0$$
shows that
\begin{equation}\label{depth R(I)}
\depth \R(I) = \left\{
     \begin{array}{ll}
       n+1 & \text{if $0\leq d-1\leq n-1$,}\\[1ex]
       n & \text{if $d-1\geq n$.}
     \end{array}
   \right.
\end{equation}
As $\dim \R(I) = n+1$, it follows that $\R(I)$ is almost Cohen-Macaulay and is Cohen-Macaulay if and only if $d-1\leq n-1$, i.e. $d\leq n$.
\end{proof}


\section{Generalized de Jonqui\`eres parameterizations}\label{Generalized de Jonq section}

With our treatment of the de Jonqui\`eres transformations of \cite{HS14} concluded, we now consider the \textit{generalized} de Jonqui\`eres transformations $\deJonq:\,\P^n\dashrightarrow\P^n$ introduced and studied by Ramos and Simis in \cite{RS22}. In the case of identity Cremona support, the defining ideal of the Rees ring is known and may be realized as the ideal of a downgraded sequence \cite[2.6]{RS22}. However, methods different from those of the previous sections are used to show generation. Moreover, the Cohen-Macaulayness of the Rees ring was left as a question in this instance \cite[2.10]{RS22}. The objective of this section is to show that the previous tools generalize to this new setting, and to determine the Cohen-Macaulayness of the Rees algebra of the ideal of a generalized de Jonqui\`eres transformation, and settle a conjecture of Ramos and Simis \cite[2.10]{RS22}.

Similar to the notion of a de Jonqui\`eres transformation in \Cref{de Jonquieres defn}, a generalized de Jonqui\`eres transformation is constructed from a supporting Cremona map. Let $R'=k[x_1,\ldots, x_n]$ denote the coordinate ring of $\P^{n-1}$, and let $\Crem:\,\P^{n-1}\dashrightarrow \P^{n-1}$ be a Cremona transformation defined by homogeneous forms $g_1,\ldots,g_n$ of degree $\Cremdeg \geq 1$ in $R'$.

\begin{defn}[{\cite{RS22}}]\label{generalized de Jonquieres defn}
With the Cremona map $\Crem:\,\P^{n-1}\dashrightarrow \P^{n-1}$ above, let $f,g \in R=k[x_1,\ldots,x_{n+1}]$ be two homogeneous polynomials with $\deg g = \deg f+\Cremdeg$. Moreover, assume that $f$ and $g$ are $x_{n+1}$-monoids with at least one involving $x_{n+1}$, and that $\gcd(f,g) =1$. The rational map $\deJonq:\, \P^n\dashrightarrow \P^n$ defined by $fg_1,\ldots,fg_n,g$ will be called a \textit{generalized de Jonqui\`eres transformation}, with Cremona support $\Crem$.
\end{defn}

Here the terminology that $f$ and $g$ are $x_{n+1}$-\textit{monoids} simply means that they are polynomials of degree one in $x_{n+1}$, i.e. as polynomials in $R[x_{n+1}]$ regraded appropriately. As such, we may write 
\begin{equation}\label{f,g decomp}
    \left\{
     \begin{array}{l}
       f=f_0+f_1x_{n+1},\\[1ex]
      g=g_0+g_1x_{n+1},
     \end{array}
   \right.
\end{equation}
 for homogeneous polynomials $f_0,f_1,g_0,g_1 \in R'$. We note that this condition guarantees that $\mathfrak{J}$ is birational \cite[1.6]{RS22}. In particular, a generalized de Jonqui\`eres transformation is itself a Cremona transformation, supported by $\Crem$ in one dimension lower.

With the conventions above, we now state the setting for the duration of this section, noting that it agrees with the situation of \cite[\textsection 3.2]{RS22}. 

\begin{set}[{\cite{RS22}}]\label{generalized de Jonq setting}
Let $R=k[x_1,\ldots,x_{n+1}]$ with subring $R'=k[x_1,\ldots,x_n]$ and $\m=(x_1,\ldots,x_n)$ its homogeneous maximal ideal. Let $f$ and $g$ be $x_{n+1}$-monoids, at least one of which involves $x_{n+1}$, with $\gcd(f,g)=1$. Assume that $\deg g=d$ and $\deg f=d-1$ for some $d\geq 2$, and let $I= (f\m,g)= (fx_1,\ldots,fx_n,g)$.
\end{set}

 With the assumption that $g$ is $x_{n+1}$-monoidal of degree $d\geq 2$, from (\ref{f,g decomp}) it follows that $g\in \m^{d-1}$. With this observation, one may produce a presentation of $I$ from the mapping cone construction.

\begin{prop}[{\cite[2.1]{RS22}}] \label{Presentation of I - generalized de Jonq}
With respect to the generating set $\{fx_1,\ldots,fx_n,g\}$, the ideal $I$ may be presented by 
$$R^{\binom{n}{2} +1} \overset{\varphi}{\longrightarrow} R^{n+1} \rightarrow I \rightarrow 0$$ 
where
\[
\varphi=\left(
\begin{array}{ccc|c}
   &  & & \\
   & \delta_1(\x')& & \partial g\\
   & & & \\
   \hline
   0& \cdots & 0 &-f 
\end{array}
\right)
\]
where $\delta_1(\x')$ is the first Koszul differential on $\x'= x_1,\ldots,x_n$, and $\partial g$ is a column of homogeneous elements, of the same degree, with $[x_1 \ldots x_n]\cdot \partial g= g$.
\end{prop}

We note that the presentation of the ideal $I$ is nearly identical to the presentation of the ideal of a de Jonqui\`eres transformation in \Cref{Presentation of I} (\cite[2.2]{HS14}). However, we note that the fact that $g\in \m R$ is crucial to this construction, so that $\partial g$ exists.

\begin{rem}\label{gen de Jonq S(I) remark}
As before, we may use the presentation matrix $\varphi$ to describe the symmetric algebra of $I$. Indeed, the symmetric algebra is $\S(I) \cong R[y_1,\ldots,y_{n+1}]/\L$, where $\L= ([y_1\ldots y_{n+1}]\cdot \varphi)$. Noting that $\delta_1(\x')$ is the first Koszul differential on $\x'= x_1,\ldots,x_n$, similar to the observation in \Cref{S(I) remark} it follows that $\L = I_2(\psi) +(h)$ where $\psi$ is as before in (\ref{psi defn}) and $h$ is the product of $[y_1 \ldots y_{n+1}]$ and the last column of $\varphi$.
\end{rem}

Notice that $R=k[x_1,\ldots,x_{n+1}]$, yet the matrix $\psi$ only involves the indeterminates $x_1,\ldots,x_n$. As such, the constructions presented here will be analogous, and in most cases identical, to those of the previous sections. Moreover, as $I=(f\m,g)$, the ideal $\m=(x_1,\ldots,x_n)$ will play a similar role as before, but we must proceed with caution, as this is not the maximal $R$-ideal. Additionally, we must take care as the polynomial $h$ does involve the last indeterminate $x_{n+1}$. In particular, from bidegree considerations and (\ref{f,g decomp}), we note that $h \in \m^{d-2}$ yet it has bidegree $(d-1,1)$.

\begin{rem}\label{gen de Jonq downgraded sequence remark}
We may employ the method of downgraded sequences as in \Cref{algorithm defn}, where $\partial-$ denotes a column as in \Cref{del notation} with respect to $x_1,\ldots,x_n$, the generators of $\m$. As noted, $h\in \m^{d-2}$ and so this sequence has only $d-1$ iterations beginning with $h_1=h$. Writing $h_1,\ldots,h_{d-1}$ to denote a full downgraded sequence, we note that each $h_i$ is of bidegree $(d-i,i)$. Moreover, for each $1\leq i\leq d-1$, we write $\J_i$ to denote the ideal $\J_i=I_2(\psi)+(h_1,\ldots,h_i)$. Repeating the arguments of \Cref{Algorithm Section} shows these ideals are well defined.
\end{rem}

As we will see, the fact that the length of a full downgraded sequence is one less than the degree of $h$, with respect to $x_1,\ldots,x_{n+1}$, will account for many of the figures in this section differing from those in the previous sections.

\begin{ex}\label{ex gen de Jonq}
Let $R=k[x_1,x_2,x_3,x_4]$ and let $\m=(x_1,x_2,x_3)$ denote the maximal ideal of the subring $R'=k[x_1,x_2,x_3]$. Let $f=x_1^2x_4$ and $g=x_1^2x_2^2+x_3^3x_4$, and consider the ideal $I=(f\m,g)=(x_1^3x_4, x_1^2x_2x_4, x_1^2x_3x_4, x_1^2x_2^2+x_3^3x_4)$. As $f$ and $g$ are $x_4$-monoids, it follows that $I$ satisfies the assumptions of \Cref{generalized de Jonq setting}. Following the construction in \Cref{Presentation of I - generalized de Jonq}, one may take $h=x_1x_2^2y_1+x_3^2x_4y_3-x_1^2x_4y_4$. With this and \Cref{algorithm defn}, a downgraded sequence of $h$, with respect to $x_1,x_2,x_3$, may be taken as
\[
 \begin{array}{ccccc}
  h_1=x_1x_2^2y_1+x_3^2x_4y_3-x_1^2x_4y_4, &\quad &  h_2=x_2^2y_1^2+x_3x_4y_3^2-x_1x_4y_1y_4, &\quad  & h_3=x_2y_1^2y_2+x_4y_3^3-x_4y_1^2y_4. 
\end{array}
\]
\end{ex}

As noted in \Cref{Algorithm Section}, this algorithm may be realized as an exchange process for the variables $x_1,\ldots,x_n$ in each monomial of the downgraded polynomials at each step. Since $f$ and $g$ are $x_{n+1}$-monoids, following (\ref{f,g decomp}) there can be only $d-1$ iterations in the setting of a generalized de Jonqui\`eres transformation. Indeed, $x_{n+1}$ can never be exchanged, similar to the observations in \Cref{gen de Jonq downgraded sequence remark}. As such, the algorithm terminates at $h_3$ in \Cref{ex gen de Jonq} when $d=4$, differing from the behavior in the previous setting of \Cref{Algorithm Section}.

\begin{rem}\label{defining ideal already known}
With the assumptions of \Cref{generalized de Jonq setting}, the defining ideal $\J$ of the Rees ring $\R(I)$ is known \cite[2.6]{RS22}, and may be realized as $\J=\J_{d-1}$, where $\J_{d-1}=I_2(\psi)+(h_1,\ldots,h_{d-1})$ as above. This was proven using different methods, but we note that the techniques of the previous sections may be adapted to this setting to show this. However, as our main objective is to determine the Cohen-Macaulayness of $\R(I)$ and address the conjecture in \cite[2.10]{RS22}, we forgo much of this and only introduce previous constructions as necessary.
\end{rem}

Interestingly, $I$ is an ideal that is not of linear type (unless $d=2$), yet the defining ideal of its Rees ring has no fiber equation, i.e. a defining equation of bidegree $(0,*)$. Thus $\J\subseteq \n=(x_1,\ldots,x_{n+1})$, and so the fiber ring is $\F(I) \cong k[y_1,\ldots,y_{n+1}]$. However, this is also clear as $\deJonq:\, \P^n\dashrightarrow \P^n$ is a Cremona map, and $\F(I)$ is the coordinate ring of its closed image.

\begin{prop}\label{gen de Jonq - J a colon ideal}
With the assumptions of \Cref{generalized de Jonq setting}, we have $\J = \L:\m^{d-2}$.
\end{prop}

\begin{proof}
Certainly we have $\L:\m^{d-2} \subseteq \J$, as $\J$ is a prime ideal with $\L\subseteq \J$ and $\m\nsubseteq \J$. Moreover, as noted in \Cref{defining ideal already known}, we have $\J=I_2(\psi)+(h_1,\ldots,h_{d-1})$ by \cite[2.6]{RS22}. Repeating the argument in the proof of \Cref{Ji contained in colons} shows that $\J=I_2(\psi)+(h_1,\ldots,h_{d-1})\subseteq \L:\m^{d-2}$, giving the reverse containment.
\end{proof}

It is not difficult to show that $\J=\L:\m^{\infty}$ by modifying the arguments in \Cref{J a saturation}, however great care must be taken to show this saturation agrees with the ideal in \Cref{gen de Jonq - J a colon ideal}. Indeed, the obstacle lies in that $\L$ is an ideal of $R[y_1,\ldots,y_{n+1}]$, but $\m$ is not the homogeneous maximal $R$-ideal. Hence one cannot use \cite[3.8]{KPU20} directly as before to bound the generation degree of $H_\m^0(\S(I))$. One may employ other tools from local cohomology and more general results from \cite{KPU20}, but for the sake of brevity we opt for the short proof above, using that the generators of $\J$ are already known.


\subsection{Cohen-Macaulayness and generalized de Jonqui\`eres maps}

As before, we note that $f\m\subseteq I$, and so we may factor the natural map $R[y_1,\ldots,y_{n+1}]\rightarrow \R(I)$ through $A=\R(f\m)[y_{n+1}]$, exactly as before in (\ref{R(m) -> R(I) diagram}). Here we note that $\dim R=n+1$, and so $\dim \R(I) = n+2$ and $\dim A = n+3$, differing slightly from before. However, the kernel of the induced map $A\rightarrow \R(I)$ is thus prime of height one, exactly as in \Cref{approximation section}. Moreover, as noted in \Cref{R(fm) and R(m) isomorphic remark}, there is an isomorphism $\R(f\m) \cong \R(\m)$, and since $\m$ is generated by a regular sequence, we have $\R(\m) \cong R[y_1,\ldots,y_{n}]/I_2(\psi)$, where $\psi$ is as in (\ref{psi defn}).

\begin{rem}\label{Everything basically the same}
With the observation above, we may write $\overline{\,\cdot\,}$ to denote images modulo $I_2(\psi)$ once again, in the ring $A=R(f\m)[y_{n+1}]$. Moreover, we take $\K$ to denote the $R[y_1,\ldots,y_{n+1}]$-ideal $\K = I_2(\psi)+(x_n,y_n)$, exactly as before.

We note that $\K \subseteq R'[y_1,\ldots,y_n]$ and $\m\subseteq R'$, hence some care must be taken. However, since $R'\subseteq R$ is a free extension, the arguments of \Cref{approximation section} may easily be repeated to show that the ideals $\K$ and $\m$ retain their properties outlined in the previous sections. In particular, $\overline{\K}$ is strongly Cohen-Macaulay, of linear type, and $\overline{\K}^i = \overline{\K}^{(i)}$ for all $i$ by \Cref{Kbar SCM and generically a CI}, \Cref{Kbar linear type}, and the proof of \Cref{defining ideal theorem}. Moreover, the powers of $\overline{\K}$ and $\overline{\m}$ are linked, precisely as in \Cref{linkage prop}.
\end{rem}

With the observations of \Cref{Everything basically the same}, the remainder of this section is much more streamlined, with much of the necessary tools developed in previous sections. However, we are careful to note the few differences between the arguments of this setting and those of the previous one. In particular, we may address the Cohen-Macaulayness of $\R(I)$ after producing the following description of $\overline{\J}$, noting that $\R(I) \cong A/\overline{\J}$.

\begin{prop}\label{gen de Jonq - Jbar divisorial}
With the assumptions of \Cref{generalized de Jonq setting}, we have $\overline{\J} = \frac{\overline{h}\overline{\K}^{d-2} }{\overline{x_n}^{d-2}}$.
\end{prop}

\begin{proof}
One proceeds exactly as in the proof of \Cref{J a colon ideal} to show that $\overline{\J} = \frac{\overline{h}\overline{\K}^{(d-2)} }{\overline{x_n}^{d-2}}$, noting that $\J=\L:\m^{d-2}$ by \Cref{gen de Jonq - J a colon ideal}, where $\L= I_2(\psi)+(h)$ is the defining ideal of $\S(I)$, as noted in \Cref{gen de Jonq S(I) remark}. The only item that needs to be justified in a manner different from before is that $\overline{\J} \nsubseteq \overline{\m}$, as $\m$ is not the maximal $R$-ideal here. However, this is clear as both ideals are prime of height one and $\overline{\J} \neq \overline{\m}$. Then one uses that the powers and symbolic powers of $\overline{\K}$ agree, as noted in \Cref{Everything basically the same} and the proof of \Cref{defining ideal theorem}.
\end{proof}

With the description of $\overline{\J}$ in \Cref{gen de Jonq - Jbar divisorial}, we may now complete the objective of this section and determine the Cohen-Macaulayness of $\R(I)$. From \cite[2.9]{RS22} it is known that $\R(I)$ is almost Cohen-Macaulay, but left as a conjecture as to when $\R(I)$ is Cohen-Macaulay \cite[2.10]{RS22}. As such, we conclude this section and the article by answering this question, as follows. Once again, as noted in \Cref{Everything basically the same}, all properties of $\overline{\K}$ hold as in the previous sections, despite the slightly different assumptions of \Cref{generalized de Jonq setting}.

\begin{thm}\label{gen de Jonq - Cohen-Macaulayness of Rees algebra}
With the assumptions of \Cref{generalized de Jonq setting}, $\R(I)$ is almost Cohen-Macaulay, and is Cohen-Macaulay if and only if $d\leq n+1$. 
\end{thm}

\begin{proof}
With the observations of \Cref{Everything basically the same}, by repeating the argument of the proof of \Cref{Cohen-Macaulayness of Rees algebra} we see that 
\begin{equation}\label{gen de Jonq - depth Kbar^i}
\depth \overline{\K}^i = \left\{
     \begin{array}{ll}
       n+3 & \text{if $0\leq i\leq n-1$,}\\[1ex]
       n+2 & \text{if $i\geq n$.}
     \end{array}
   \right.
\end{equation}
Here we note that $\dim R =n+1$ in this setting, hence the difference from (\ref{depth Kbar^i}) after using the Auslander-Buchsbaum formula. From \Cref{gen de Jonq - Jbar divisorial}, it follows that there is an isomorphism $\overline{\J} \cong \overline{\K}^{d-2}$. With this and (\ref{gen de Jonq - depth Kbar^i}), by comparing depths along the short exact sequence 
$$0\rightarrow \overline{\J} \rightarrow A\rightarrow \R(I) \rightarrow 0,$$
we have
\begin{equation}\label{gen de Jonq depth R(I)}
\depth \R(I) = \left\{
     \begin{array}{ll}
       n+2 & \text{if $0\leq d-2\leq n-1$,}\\[1ex]
       n+1 & \text{if $d-2\geq n$.}
     \end{array}
   \right.
\end{equation}
Again noting that $\dim R =n+1$, we have $\dim \R(I) =n+2$. Thus it follows that $\R(I)$ is almost Cohen-Macaulay, and is Cohen-Macaulay if and only if $d-2\leq n-1$, i.e. $d\leq n+1$.
\end{proof}

As such, \Cref{gen de Jonq - Cohen-Macaulayness of Rees algebra} answers Conjecture 2.10 of \cite{RS22} in the affirmative.



\begin{thebibliography}{99}




 
  
\bibitem{BM16} J.~A.~Boswell and V.~Mukundan, \textit{Rees algebras of almost linearly presented ideals}, J. Algebra \textbf{460} (2016), 102--127.

\bibitem{BH93} W. Bruns and J. Herzog, \textit{Cohen-{M}acaulay rings}, Cambridge University Press, Cambridge, 1993. 




  

 
\bibitem{CPW23} A.~Costantini, E.~F.~Price, and M.~Weaver, \textit{On Rees algebras of linearly presented ideals and modules}, to appear in Collect. Math. \texttt{https://doi.org/10.1007/s13348-024-00440-0}, \texttt{arxiv:2308.16010}. 


\bibitem{CPW24} A.~Costantini, E.~F.~Price, and M.~Weaver, \textit{On Rees algebras of ideals and modules with weak residual conditions}, preprint: \texttt{arxiv:2409.14238}.

\bibitem{CHW08} D.~Cox, J.~W. Hoffman, and H.~Wang, \textit{Syzygies and the {R}ees algebra}, J. Pure Appl. Algebra \textbf{212} (2008), 1787--1796.

\bibitem{DHS12} A.~Doria, S. H.~Hassanzadeh, and A.~Simis, \textit{A characteristic-free criterion of birationality}, Adv. Math. \textbf{230} (2012), 390--413.
  


\bibitem{Eisenbud} D.~Eisenbud, \textit{Commutative Algebra With a View Towards Algebraic Geometry}, Grad. Texts in Math. \textbf{150}, Springer-Verlag, New York, 1995. 


\bibitem{Macaulay2} D. R. Grayson and M. E. Stillman, Macaulay2, a software system for research in algebraic geometry. Available at http://www.macaulay2.com


\bibitem{HS14} S. H.~Hassanzadeh and A.~Simis \textit{Implicitization of de Jonqui\`eres parametrizations}, J. Commut. Algebra \textbf{6} (2014), 149--172.


\bibitem{HSV82} J.~Herzog, A.~Simis, and W.~V. Vasconcelos, \textit{Approximation complexes of blowing-up rings}, J. Algebra \textbf{74} (1982), 466--493.

\bibitem{HSV83} J.~Herzog, A.~Simis, and W.~V. Vasconcelos, \textit{Approximation complexes of blowing-up rings {II}}, J. Algebra \textbf{82} (1983), 53--83.
  

\bibitem{Huneke82} C. Huneke, \textit{Linkage and the Koszul homology of ideals}, Amer. J. Math. \textbf{104} (1982), 1043--1062.

\bibitem{Huneke83} C.~Huneke, \textit{Strongly Cohen-Macaulay schemes and residual intersections}, Trans. Amer. Math. Soc. \textbf{277} (1983), 739--763. 

\bibitem{HR86} C.~Huneke and M.~Rossi, \textit{The dimension and components of symmetric algebras}, J. Algebra \textbf{98} (1986), 200--210.   



\bibitem{KPU11} A. Kustin, C. Polini and B. Ulrich, \textit{Rational normal scrolls and the defining equations of Rees algebras}, J. Reine Angew. Math. \textbf{650} (2011), 23--65.

\bibitem{KPU17} A.~Kustin, C.~Polini and B.~Ulrich, \textit{The equations defining blowup algebras of height three Gorenstein ideals}, Algebra Number Theory \textbf{11} (2017), 1489--1525.

\bibitem{KPU20} A. Kustin, C. Polini and B. Ulrich, \textit{Degree bounds for local cohomology}, Proc. Lond. Math. Soc. \textbf{121} (2020), 1251--1267.


  



\bibitem{Micali64} A.~Micali \textit{Sur les alg\`ebres universelles}, Ann. Inst. Fourier (Grenoble) \textbf{14} (1964), 33--87.

\bibitem{Morey96} S. Morey, \textit{Equations of blowups of ideals of codimension two and three}, J. Pure Appl. Algebra \textbf{109} (1996), 197--211.  
  
\bibitem{MU96} S.~Morey and B.~Ulrich, \textit{Rees algebras of ideals of low codimension}, Proc. Amer. Math. Soc. \textbf{124} (1996), 3653--3661.
  
\bibitem{Nguyen14} P.~H.~L. Nguyen, \textit{On Rees algebras of linearly presented ideals}, J. Algebra \textbf{420} (2014), 186–-200.

\bibitem{Nguyen17} P.~H.~L. Nguyen, \textit{On Rees algebras of linearly presented ideals in three variables}, J. Pure Appl. Algebra \textbf{221} (2017), 2180–-2198.

\bibitem{Pan01} I.~Pan \textit{Les transformations de {C}remona stellaires}, Proc. Amer. Math. Soc. \textbf{129} (2001), 1257--1262.

\bibitem{PS15} I.~Pan and A.~Simis \textit{Cremona maps of de {J}onqui\`eres type}, Canad. J. Math. \textbf{67} (2015), 923--941.

\bibitem{RS22} Z.~Ramos and A.~Simis \textit{De Jonqui\`eres transformations in arbitrary dimension. An ideal theoretic view}, Comm. Algebra \textbf{50} (2022), 4095--4108.

\bibitem{Simis04} A.~Simis, \textit{Cremona transformations and some related algebras}, J. Algebra \textbf{280} (2004), 162--179.

\bibitem{SV81}A.~Simis and W.~V. Vasconcelos, \textit{The syzygies of the conormal module}, Amer. J. Math. \textbf{103} (1981), 203--224.
  


  


\bibitem{Trung98} N.~V. Trung, \textit{The Castelnuovo regularity of the Rees algebra and the associated graded ring}, Trans. Amer. Math. Soc. \textbf{350} (1998), 2813--2832.


  



\bibitem{VasconcelosBook94} W.~V. Vasconcelos, \textit{Arithmetic of blowup algebras}, London Math. Soc. Lecture Note Ser. \textbf{195}, Cambridge University Press, Cambridge, 1994.  


\bibitem{Weaver23} M. Weaver, \textit{On {R}ees algebras of ideals and modules over hypersurface rings},  J. Algebra \textbf{636} (2023), 417--454.

\bibitem{Weaver24} M. Weaver, \textit{The equations of {R}ees algebras of height three {G}orenstein ideals in hypersurface rings}, J. Commut. Algebra \textbf{16} (2024), 123--149.



\bibitem{Weaver25} M. Weaver, \textit{Blowups of hypersurfaces}, preprint: \texttt{arxiv:2507.10322}.










 
\end{thebibliography}
\end{document}